\documentclass[12pt,amssymb]{amsart}
\usepackage{amssymb}
\usepackage[mathscr]{eucal}
\usepackage[all,cmtip]{xy}

\newcommand{\Z}{{\mathbb Z}}

\newcommand{\Ga}{\mathrm{Gal}}
\newcommand{\he}{H_{\mathrm{\acute{e}t}}}
\newtheorem{thm}{Theorem}[section]
\newtheorem{lemma}[thm]{Lemma}
\newtheorem{prop}[thm]{Proposition}
\newtheorem{cor}[thm]{Corollary}

\font\brus=wncyr10.240pk scaled 1200 .240pk
\DeclareFontFamily{U}{wncy}{}
    \DeclareFontShape{U}{wncy}{m}{n}{<->wncyr10}{}
    \DeclareSymbolFont{mcy}{U}{wncy}{m}{n}
    \DeclareMathSymbol{\Sha}{\mathord}{mcy}{"58}
\begin{document}

\title[Finiteness properties]{Properness of the global-to-local map for algebraic groups with toric connected component and other finiteness properties}

\author[A.~Rapinchuk]{Andrei S. Rapinchuk}

\author[I.~Rapinchuk]{Igor A. Rapinchuk}

\address{Department of Mathematics, University of Virginia,
Charlottesville, VA 22904-4137, USA}

\email{asr3x@virginia.edu}

\address{Department of Mathematics, Michigan State University, East Lansing, MI
48824, USA}

\email{rapinchu@msu.edu}

\begin{abstract}
This is a companion paper to \cite{RR-tori}, where we proved the finiteness of the Tate-Shafarevich group for an arbitrary torus $T$ over a finitely
generated field $K$ with respect to any divisorial set $V$ of places of $K$. Here, we extend this result to any $K$-group $D$ whose connected component
is a torus (for the same $V$), and as a consequence obtain a finiteness result for the local-to-global conjugacy of maximal tori in reductive groups over
finitely generated fields. Moreover, we prove the finiteness of the Tate-Shafarevich group for tori over function fields $K$ of normal varieties defined over
base fields of characteristic zero and satisfying Serre's condition (F), in which case $V$ consists of the discrete valuations associated with the prime divisors on the variety (geometric places). In this situation, we also establish the finiteness of the number of $K$-isomorphism classes of algebraic $K$-tori of a given dimension having good reduction at all $v \in V$, and then discuss ways of extending this result to positive characteristic. Finally, we prove the finiteness of the number of isomorphism classes of forms of an absolutely almost simple group defined over the function field of a complex surface that have good reduction at all geometric places.
\end{abstract}

\maketitle

\hfill \parbox[t]{6.5cm}{\it In memory of Stepan F. Rapinchuk, our father and grandfather.}


\section{Introduction}

Our goal in this paper is to extend and generalize in different directions some of the results obtained in \cite{RR-tori}. First, let $K$ be a finitely generated field, and let $V$ be a divisorial set of discrete valuations of $K$. This means that  $V$ consists of the discrete valuations corresponding to the prime divisors on a separated normal scheme $\mathfrak{X}$ of finite type over $\mathbb{Z}$ with function field $K$, which we will refer to as a {\it model} of $K$ (cf. the discussion in \cite[5.3]{RR-survey}). For a $K$-defined linear algebraic group $G$, one considers the global-to-local map in Galois cohomology
$$
\lambda_{G , V} \colon H^1(K , G) \longrightarrow \prod_{v \in V} H^1(K_v , G)
$$
(we refer the reader to \cite[Ch. III]{Serre-GC} and \cite[4.2]{RR-survey} for the basic notions and standard notations pertaining to the Galois cohomology of algebraic groups). It has been known for a long time (see \cite{BS}) that if $K$ is a number field, then the map $\lambda_{G , V}$ is {\it proper} (i.e., the preimage of any finite set is finite) for every algebraic $K$-group $G$ and any divisorial set $V$ (which in this situation simply means that $V$ contains almost all nonarchimedean places of $K$). Recent work, however, has led to the expectation that $\lambda_{G , V}$ should be proper much more generally; in particular, this was conjectured  to be the case for any (connected) reductive $K$-group $G$ and any divisorial set $V$ (see \cite[Conjecture 6.1]{RR-survey})\footnotemark \footnotetext{In a very broad sense, this extension is in line with the general discussion by Lang (\cite[p. 202]{LangHyperbolic}) of whether finiteness results for the number of rational points on certain varieties defined over number fields should in fact hold over all finitely generated fields of characteristic zero.}. In \cite{RR-tori}, this conjecture was proved for any $K$-torus $T$ --- note that in this case, the properness of $\lambda_{T , V}$ is equivalent to the finiteness of the Tate-Shafarevich group $\text{\brus{ Sh}}(T , V) = \ker \lambda_{T , V}$.  In the present paper, we extend this result to $K$-groups whose connected component is a torus; it should be pointed out that the passage to connected groups is a non-trivial matter even in the number field situation (cf. \cite[\S 7]{BS}), and the general case requires new tools. More precisely, we will prove the following statements.

\begin{thm}\label{T:1}
Let $K$ be a finitely generated field and $V$ be a divisorial set of places of $K$. Then for any linear algebraic $K$-group $D$ whose connected component $D^{\circ}$ is a torus and any finite Galois extension $L/K$, we have:

\vskip2mm

\noindent \ {\rm (i)} \parbox[t]{12cm}{if $D$ is commutative, then  the map $$\lambda^i_{D , V , L/K} \colon H^i(L/K , D) \longrightarrow \prod_{v \in V} H^i(L_w/K_v , D), \ \ w \vert v$$ is proper for any $i \geq 1$;}

\vskip2mm

\noindent {\rm (ii)} \parbox[t]{12cm}{in the general case, the map $$\lambda^1_{D, V, L/K} \colon H^1(L/K , D) \longrightarrow \prod_{v \in V} H^1(L_w/K_v , D), \ \ w \vert v$$ is proper.}

\vskip2mm

\noindent $($In both parts, for each $v \in V$, we pick {\it one} extension $w$ to $L$.$)$

\end{thm}

\begin{thm}\label{T:2}
For $K$, $V$,  and $D$ as in Theorem \ref{T:1}, the map $$\lambda_{D , V} \colon H^1(K , D) \longrightarrow \prod_{v \in V} H^1(K_v , D)$$ is proper.
\end{thm}
As an application, we obtain the following finiteness result on the local-global conjugacy problem for maximal tori in any reductive group.
\begin{thm}\label{T:2A}
Let $G$ be a connected reductive group over a finitely generated field $K$, and let $V$ be a divisorial set of places of $K$. Fix a maximal $K$-torus $T$ of $G$ and let $\mathscr{C}(T)$ be the set of all maximal $K$-tori $T'$ of $G$ such that $T$ and $T'$ are $G(K_v)$-conjugate for all $v \in V$. Then $\mathscr{C}(T)$ consists of finitely many $G(K)$-conjugacy classes.
\end{thm}

\noindent (We note that over global fields, this result is due to P.~Gille and L.~Moret-Bailly \cite[Theorem 7.9]{Gil-Mor}; it was used by E.~Ullmo and A.~Yafaev \cite{Ul-Yaf} in their work on the Andr\'e-Oort conjecture.)

\vskip2mm

Second, we extend some  of the finiteness results from \cite{RR-tori} to a different class of fields. More precisely, let $K = k(X)$ be the function field of a normal irreducible variety $X$ over a field $k$ that satisfies Serre's condition (F) (see \S \ref{S:Funct} and references therein), and let $V$ be the set of discrete valuations associated with the prime divisors on $X$ (geometric places).

\begin{thm}\label{T:3}
With notations and conventions as above, if $k$ is a field of characteristic zero that satisfies condition $(\mathrm{F})$, then for any $d \geq 1$, there exist only finitely many $K$-isomorphism classes of $d$-dimensional $K$-tori that have good reduction at all $v \in V$.
\end{thm}

As we discussed in \cite[Remark 2.5]{RR-tori}, results about the finiteness of the number of $K$-isomorphism classes of algebraic $K$-tori of a given dimension $d \geq 1$ that have good reduction at all places $v \in V$  are no longer valid as stated in positive characteristic. More precisely, Theorem \ref{T:3} above and \cite[Theorem 1.1]{RR-tori} are {\it false} even for a global function field $K$ and an arbitrary divisorial set $V$. The situation can be fixed by considering some special sets $V$. We recall that a finitely generated field $K$ of positive characteristic can be presented as the function field $k(X)$ of a {\it complete} normal variety $X$ over a finite field $k$. Then the set $V$ of discrete valuations of $K$ associated with the prime divisors on $X$ will be called a {\it complete divisorial} set of places. It turns out that Theorem \ref{T:3} (hence also \cite[Theorem 1.1]{RR-tori}) extends to characteristic $p > 0$ if we assume that the divisorial set of places $V$ is complete --- see Theorem \ref{T:Z1}(i). On the other hand, this theorem remains valid for any divisorial set of places $V$ if one considers only those $K$-tori that split over an extension $K_T/K$ of degree prime to $p$ --- see Theorem \ref{T:Z1}(ii). These results are derived from a more general statement (see Theorem \ref{T:tori-GR}) that also subsumes the essential part of the proof of \cite[Theorem 1.1]{RR-tori}.

Next, over function fields $K = k(X)$ as above, we have the following finiteness result for the Tate-Shafarevich group.
\begin{thm}\label{T:4}
If $k$ has characteristic zero and is of type $(\mathrm{F})$, then for any $K$-torus $T$, the Tate-Shafarevich group $\text{\brus{ Sh}}(T , V)$ is finite.
\end{thm}

We also completely resolve the question of the properness of the global-to-local map for finite Galois modules over function fields $K = k(X)$ in all characteristics without any assumptions on the base field --- see Proposition \ref{P:Finite2} and Remark 6.3.


Finally, we have the following finiteness result for forms of absolutely almost simple groups with good reduction over the function fields of
complex surfaces.
\begin{thm}\label{T:CSurf}
Let $K = k(S)$ be the function field of a smooth surface $S$ over an algebraically closed field $k$ of characteristic zero, and let $V$ be the set of discrete valuations of $K$ associated with the prime divisors of $S$. Then for any absolutely almost simple simply connected algebraic $K$-group $G$, the set of $K$-isomorphism classes of $K$-forms $G'$ of $G$ that have good reduction at all $v \in V$ is finite.
\end{thm}


\section{Theorem \ref{T:2}: the case of a finite group}

The goal of this section is to prove the following.

\begin{prop}\label{P:1}
Assume that $K$ is a finitely generated field equipped with a divisorial set of places $V$, and
let $\Omega$ be a finite (but not necessarily commutative) Galois module\footnotemark. Then the map
$$
H^1(K , \Omega) \stackrel{\kappa}{\longrightarrow} \prod_{v \in V} H^1(K_v , \Omega)
$$
is proper.
\end{prop}

\footnotetext{I.e., a finite group with a continuous action of the absolute Galois group $\mathrm{Gal}(K^{\mathrm{sep}}/K)$.}

The proof relies on the following (known) result, which is of independent interest.
\begin{prop}\label{P:2}
Let $K$ be an infinite finitely generated field and $V$ be a divisorial set of places of $K$. If a finite separable extension $L/K$ satisfies $L_w = K_v$ for all $v \in V$ and $w \vert v$, then $L = K$.
\end{prop}
\begin{proof}
Without loss of generality, we may assume that $L/K$ is a Galois extension; let $\mathscr{G} = \mathrm{Gal}(L/K)$ be its Galois group.

First, let us consider the case where $K$ is a global field; then $V$ consists of almost all nonarchimedean places of $K$. Assume that $L \neq K$ and pick a nontrivial automorphism $\sigma \in \mathscr{G}$. It follows from Chebotarev's Density Theorem (cf. \cite[Ch. VII, 2.4]{ANT}) that there exists $v \in V$ such that some for $w \vert v$, the extension $L_w/K_v$ is unramified and has $\sigma$ as its Frobenius automorphism. Then in particular $L_w \neq K_v$, a contradiction.

We now turn to the situation where the field $K$ is not global and first consider the case where $K$ has characteristic zero. We may assume that $V$ consists of the discrete valuations corresponding to the prime divisors on a model $\mathfrak{X} = \mathrm{Spec}\: A$, where $A$ is a finitely generated integrally closed $\mathbb{Z}$-algebra with fraction field $K$. Since $K$ is not global, it has transcendence degree $r \geq 1$ over $\mathbb{Q}$, so one can find $x_1, \ldots , x_r \in A$ that form a transcendence basis of $K/\mathbb{Q}$. Set $B = \mathbb{Z}[x_1, \ldots , x_r]$ and $F = \mathbb{Q}(x_1, \ldots , x_r)$. We can find a nonzero polynomial $h \in \mathbb{Z}[x_1, \ldots , x_r]$ such that the localization $A_h$ is integral over $B_h$.

Next, pick a primitive element $\alpha$ of $L$ over $K$ that is integral over $B$. We can write its minimal polynomial $g \in B[y]$ over $K$ as $g = g(x_1, \ldots , x_r, y) \in \mathbb{Z}[x_1,\ldots , x_r, y]$.
Since $g$ is irreducible in $\mathbb{Z}[x_1, \ldots , x_r, y]$, by the version of Hilbert's Irreducibility Theorem given in \cite[13.4]{FJ}, there exists an $r$-tuple $(x_1^0, \ldots , x_r^0) \in \mathbb{Z}^r$ such that $h(x_1^0, \ldots , x_r^0) \neq 0$ and the polynomial $g(x_1^0, \ldots , x_r^0, y) \in \mathbb{Q}[y]$ is irreducible. Then the polynomial $$\varphi := g(x_1^0, x_2, \ldots , x_r) \in \mathbb{Q}[x_2, \ldots , x_r, y]$$ is also irreducible. Indeed, the leading coefficient of $\varphi$ as a polynomial in $y$ is 1, so $\varphi$ has content 1 as a polynomial in $\mathbb{Q}[x_2, \ldots , x_r][y]$. This means that  any  factor in a possible factorization of $\varphi$ must have a positive $y$-degree. Since  the specialization $\varphi(x_2^0, \ldots , x_r^0, y) = g(x_1^0, \ldots , x_r^0, y)$ is irreducible, $\varphi$ itself is irreducible.

Let $v_0$ be a discrete valuation of $F$ corresponding to the irreducible polynomial $x_1 - x_1^0$. Since $h(x_1^0, \ldots , x_r^0) \neq 0$, we see that $h$ is relatively prime to $x_1 - x_1^0$, so $B_h$ is contained in the valuation ring of $v_0$. As $A_h$ is integral over $B_h$, we conclude that $A_h$ is contained in the valuation ring of an extension $v$ of $v_0$ to $K$, hence $v \in V$. Furthermore, $\alpha$ is contained in the valuation ring of an extension $w$ of $v$ to $L$. We can view $\varphi$ as a polynomial over the residue field $F^{(v_0)} = \mathbb{Q}(x_2, \ldots , x_r)$. We note that the image $\bar{\alpha}$ of $\alpha$ in the residue field $L^{(w)}$ satisfies $\varphi$, and since $\varphi$ is irreducible of degree $\deg_y \varphi = \deg_y g$, we conclude that the residual degree $f(w \vert v_0)$ equals $[L : F]$. As $f(w \vert v_0) = f(w \vert v) f(v \vert v_0)$ and $f(w \vert v) \leq [L : K]$ and $f(v \vert v_0) \leq [K:F]$, it follows that $f(w \vert v) = [L : K]$, and therefore $[L_w : K_v] = [L : K]$. So, our assumption that $L_w = K_v$ yields that $L = K$.

In order to treat the positive characteristic case, we observe that any finitely generated field $K$ of positive characteristic can be presented as the function field $K = k(X)$ of a geometrically integral normal variety $X$ over a finite field $k$, and then we may assume that $V$ consists of the discrete valuations associated with the prime divisors of $X$. Since the case of global fields has already been considered, we may assume that $\dim X \geq 2$. So, we conclude the proof of Proposition \ref{P:2} by applying the following more general statement.

\begin{prop}\label{P:FFF3}
Let $K = k(X)$ be the function field of a normal geometrically integral variety $X$ of dimension $\dim X \geq 2$ defined over a field $k$, and let $V$ be the set of discrete valuations of $K$ associated with the prime divisors of $X$. If $L/K$ is a finite separable extension such that $L_w = K_v$ for all $v \in V$ and $w \vert v$, then $L = K$.
\end{prop}

The argument is similar to the characteristic zero case in the proof of Proposition \ref{P:2}. Since $X$ is geometrically integral, it is, in particular, geometrically reduced, and hence $K$ is a separable extension of $k$ (cf. \cite[Lemma 10.44.1]{Stacks}). Thus, there exists a separating transcendence basis $x_1, \ldots, x_r$,
so that $K$ is a finite separable extension of $F = k(x_1, \ldots , x_r)$ (cf. \cite[2.6]{FJ}). Note that $r \geq 2$ by assumption. Replacing $X$ by an open subset (which may only reduce $V$), we may assume that $X$ is affine and $x_1, \ldots , x_r$ lie in the algebra of regular functions $A = k[X]$. Set $R = k[x_1]$ and $B = k[x_1, \ldots , x_r] = R[x_2, \ldots , x_r]$. We can find a nonzero polynomial $h \in B$ such that the localization $A_h$ is integral over $B_h$. Since, by construction, the extension $L/F$ is separable, we can choose a primitive element $\alpha$ for this extension that is integral over $B$. Let $g \in B[y]$ be the minimal polynomial of $\alpha$ over $F$. Then $g$ can be viewed as a polynomial in $R[x_2, \ldots , x_r, y]$, which is separable in $y$. According to \cite[Proposition 13.4.1]{FJ}, the ring $R$ is Hilbertian. Thus, one can find $(x_2^0, \ldots , x_r^0) \in R^{r-1}$ so that the polynomial $g(x_2^0, \ldots , x_r^0, y) \in k(x_1)[y]$ is irreducible and $h(x_2^0, \ldots , x_r^0) \neq 0$. Then the same argument as above shows that the valuation $v_0$ of $F$ associated with $x_2 - x_2^0$ extends to a valuation $v \in V$, and then for $w \vert v$ we have $[L_w : K_v] = [L : K]$, implying that $L = K$.
\end{proof}

\noindent {\bf Remark 2.4.} The assertion of Proposition \ref{P:2} was stated in geometric language by Raskind (see \cite[Lemma 1.7]{Rask}). The argument he sketches is based on the consideration of $\zeta$-functions. The above proof based on Hilbert's Irreducibility Theorem is of a somewhat more general nature as it applies to function fields of algebraic varieties over not necessarily finitely generated fields. This will be used in \S \ref{S:local-global}.

\vskip1mm

\noindent {\it Proof of Proposition \ref{P:1}.}
Using twisting (cf. \cite[Ch. I, 5.3]{Serre-GC}), we see that it is enough to establish the finiteness of $\ker \kappa$ for any finite Galois module $\Omega$. Furthermore, for any Galois extension $L/K$, we have the following inflation-restriction exact sequence in non-commutative cohomology (cf. {\it loc. cit.}, Ch.~I, 5.8):
$$
1 \to H^1(L/K , \Omega) \longrightarrow H^1(K , \Omega) \longrightarrow H^1(L , \Omega).
$$
If $L/K$ is a finite Galois extension, then the set $H^1(L/K , \Omega)$ is also finite, and it is enough to show that the kernel of the map $$H^1(L , \Omega) \to \prod_{w \in V^L} H^1(L_w , \Omega),$$ where $V^L$ consists of all extensions of places in $V$ to $L$, is finite. Thus, replacing $K$ by a suitable finite Galois extension, we may assume that $\Omega$ is a trivial Galois module over $K$. Let us show that in this case, $\ker \kappa$ is actually trivial.

Indeed, let $x \in \ker \kappa$. Since the Galois action on $\Omega$ is trivial, $x$  is represented by a continuous homomorphism $\chi \colon \mathscr{G} \to \Omega$ of the absolute Galois group $\mathscr{G} = \mathrm{Gal}(K^{\mathrm{sep}}/K)$. Let $\mathscr{H} = \ker \chi$, and let $L$ be the finite Galois extension of $K$ corresponding to $\mathscr{H}$. The fact that $x \in \ker \kappa$ then implies that $\chi$ vanishes on every decomposition group $\mathscr{G}(v) = \mathrm{Gal}(K_v^{\mathrm{sep}}/K_v)$, $v \in V$. Then $\mathscr{G}(v) \subset \mathscr{H}$, implying that $L_w = K_v$ for all $v \in V$, $w \vert v$. Proposition \ref{P:2}  then yields $L = K$, i.e. $\mathscr{H} = \mathscr{G}$, proving that $\chi$ is the trivial homomorphism. Thus, $x$ is the trivial class, as required. \hfill $\Box$

\vskip1mm

\noindent {\bf Remark 2.5.}  The map $H^1(K , \Omega) \to \prod_{v \in V} H^1(K , \Omega)$ may not be injective even when the finite Galois module $\Omega$ is commutative and $V$ is the set of all places (including archimedean ones) of a number field $K$ --- see \cite[Ch. III, 4.7]{Serre-GC}.

\vskip1mm

 \noindent {\bf Remark 2.6.} The assertion of Proposition \ref{P:2} is false when $K$ is the function field of a smooth projective irreducible curve complex curve $X$ of genus $\geq 1$ and $V$ is the set of places associated with the closed points of $X$.  Then for any $n > 1$, one can find an element  $x \in \mathrm{Pic}^0(X)$ of order precisely $n$, and let $D$ be the degree zero divisor on $X$ representing $x$. Then $nD$ is the divisor $(f)$ of a function $f \in K^{\times}$. The fact that the order of $x$ in $\mathrm{Pic}^0(X)$ is $n$ implies that the order of $f$ in the quotient $K^{\times}/{K^{\times}}^n$ is also $n$, and therefore $L = K(\sqrt[n]{f})$ is a cyclic Galois extension of $K$ of degree $n$. On the other hand, let $p \in X$ be any closed point and $v = v_p$ be the corresponding discrete valuation of $K$. Since $v(f)$ is a multiple of $n$, the extension $L/K$ is unramified at $v_p$, i.e. $L_w/K_v$ is unramified. But since the residue field of $K_v$ is $\mathbb{C}$, hence algebraically closed, the field $K_v$ does not have any nontrivial unramified extensions. Thus, $L_w = K_v$.

\section{Proofs of Theorems \ref{T:1}, \ref{T:2}, and \ref{T:2A}}\label{S:Proofs}

We begin by reformulating
the assertion of Theorem \ref{T:1} in the language of adeles. Let $K$ be a field equipped with a set $V$
of discrete valuations, and for $v \in V$, denote by $\mathcal{O}_v$ the valuation ring in the corresponding completion $K_v$. To define the adelic group associated with
a linear algebraic $K$-group $D$, we first fix a faithful $K$-defined representation $D \hookrightarrow \mathrm{GL}_n$, and then let
$$
D(\mathbb{A}(K , V)) := \left\{ \left. \ (g_v) \in \prod_{v \in V} D(K_v) \ \right\vert \ g_v \in D(\mathcal{O}_v) \ \text{for almost all} \ v \in V \  \right\},
$$
where $D(\mathcal{O}_v) = D(K_v) \cap \mathrm{GL}_n(\mathcal{O}_v).$
In the situations that are most relevant for our discussion,
the set $V$ satisfies the following property:

\vskip2mm

\noindent (A) For any $a \in K^{\times}$, the set $V(a) := \{ v \in V \ \vert \ v(a) \neq 0 \}$ is finite.

\vskip2mm

\noindent For example, this is the case for any divisorial set of places $V$ of a finitely generated field $K$. This property has several important consequences. First, it implies that the adelic group does not depend on the initial choice of a faithful $K$-defined representation $D \hookrightarrow \mathrm{GL}_n$; more precisely, a $K$-defined isomorphism between two linear $K$-groups induces an isomorphism between the corresponding adelic groups. Second, property (A)
enables us to consider the diagonal embedding $D(K) \hookrightarrow D(\mathbb{A}(K , V))$. More generally, for any finite separable field extension $L/K$, the set $V^L$ consisting of all extensions of places from $V$ to $L$, also satisfies (A), so we again have  the diagonal embedding $D(L) \hookrightarrow D(\mathbb{A}(L , V^L))$. Moreover, if $L/K$ is a Galois extension with Galois group $\mathscr{G}$, then the standard action of $\mathscr{G}$ on $D(L)$ naturally extends to an action on $D(\mathbb{A}(L , V^L))$ (cf. \cite[\S 3]{RR-tori}). We then have the following.

\begin{prop}\label{P:adeles}
Let $K$ be a finitely generated field, $V$ a divisorial set of places of $K$, and $D$ a linear algebraic $K$-group whose connected component $D^{\circ} = T$
is a torus. Given a finite Galois extension $L/K$,

\vskip1.5mm

\noindent \ {\rm (i)} \parbox[t]{16cm}{if $D$ is commutative, then  the kernel of $\lambda^i_{D, V, L/K} \colon H^i(L/K , D) \to \prod_{v \in V} H^i(L_w/K_v , D)$ coincides with the kernel of $\theta^i_{L/K} \colon H^i(L/K , D(L)) \to H^i(L/K , D(\mathbb{A}(L , V^L)))$ for all $i \geq 1$;}

\vskip1mm

\noindent {\rm (ii)} \parbox[t]{16cm}{in the general case, the kernel of the map $\lambda^1_{D, V, L/K} \colon H^1(L/K , D) \to \prod_{v \in V} H^1(L_w/K_v , D)$ coincides with the kernel of the map $\theta^1_{L/K} \colon H^1(L/K , D) \to H^1(L/K , D(\mathbb{A}(L , V^L)))$.}
\end{prop}

The proposition is an immediate consequence of the following result.

\begin{lemma}\label{L:integr}
With notations as in Proposition \ref{P:adeles}, we have the following statements:

\vskip2mm

\noindent \ {\rm (i)} \parbox[t]{15cm}{If $D$ is commutative, then for almost all $v \in V$ and $w \vert v$, the group homomorphisms
$$
\mu^i \colon H^i(L_w/K_v , D(\mathcal{O}_{L_w})) \to H^i(L_w/K_v , D(L_w))
$$
are injective for all $i \geq 1$.}

\vskip2mm

\noindent  {\rm (ii)} \parbox[t]{15cm}{In the general case, for almost all $v \in V$ and $w \vert v$, the map
$$
\mu^1 \colon H^1(L_w/K_v , D(\mathcal{O}_{L_w})) \to H^1(L_w/K_v , D(L_w))
$$
has trivial kernel.}
\end{lemma}
\begin{proof}
Let $E$ be the splitting field of the torus $T = D^{\circ}$, and set $P = EL$. Then for almost all $v \in V$ and $u \vert v$, the following two properties hold:

\vskip1.5mm

\noindent (a) the extension $P_u/K_v$ is unramified;

\vskip1mm

\noindent (b) \parbox[t]{12cm}{all co-characters $\chi \in X_*(T)$ are defined over $\mathcal{O}_{P_u}$, and consequently
$T(\mathcal{O}_{P_w})$ is a maximal bounded subgroup of $T(P_u)$.}

\vskip1.5mm

\noindent Furthermore, it was shown in \cite[p. 9-10]{RR-tori} that for almost all $v \in V$ and $w \vert v$, we have

\vskip1.5mm

\noindent (c) $D(L_w) = T(L_w) D(\mathcal{O}_{L_w})$.

\vskip1.5mm

\noindent We will now show that the validity of the three properties (a)-(c) implies the assertions of the lemma.
Let $\pi \in K_v$ be a uniformizer. Since $P_u/K_v$ is unramified, $\pi$ remains a uniformizer in $P_u$, so we have the following decomposition of the multiplicative group $P_u^{\times}$ as $\mathrm{Gal}(P_u/K_v)$-module:
$$
P_u^{\times} = \langle \pi \rangle \times U_{P_u} \simeq \mathbb{Z} \times U_{P_u},
$$
where $U_{P_u} = \mathcal{O}_{P_u}^{\times}$ is the group of units of $P_u$. As above, let $X_*(T)$ be the group of cocharacters of $T$. We then  have the following decomposition as $\mathrm{Gal}(P_u/K_v)$-modules:
$$
T(P_u) \simeq X_*(T) \otimes_{\mathbb{Z}} P_u^{\times} \simeq X_*(T) \otimes_{\mathbb{Z}}  (\Z \times U_{P_u}).
$$
Clearly, $X_*(T) \otimes_{\mathbb{Z}} U_{P_u}$ is a maximal bounded subgroup of $T(P_u)$, hence coincides with $T(\mathcal{O}_{P_u})$. Thus,
$$
T(P_u) \simeq X_*(T) \times T(\mathcal{O}_{P_u}).
$$
Taking $\mathrm{Gal}(P_u/L_w)$-fixed points, we obtain the decomposition
$$
T(L_w) \simeq \Gamma \times T(\mathcal{O}_{L_w}) \ \ \text{where} \ \ \Gamma = X_*(T)^{\mathrm{Gal}(P_u/L_w)}.
$$
Combining this with (c), we thus have
\begin{equation}\label{E:X1}
D(L_w) = \Gamma  D(\mathcal{O}_{L_w}),
\end{equation}
noting that $\Gamma \cap D(\mathcal{O}_{L_w}) = \{ 1 \}$ since $\Gamma$ does not contain any bounded subgroups.
If $D$ is commutative, then
(\ref{E:X1}) is actually a direct product decomposition, which immediately implies part (i) of the lemma. To prove part (ii) in the general case, let us assume
that a cocycle $\xi \in Z^1(L_w/K_v , D(\mathcal{O}_{L_w}))$ represents an element of $\ker \mu^1$. Then there exists $x \in D(L_w)$ such that
$$
\xi(\sigma) = x^{-1} \sigma(x) \ \ \text{for all} \ \ \sigma \in \mathrm{Gal}(L_w/K_v).
$$
Using (\ref{E:X1}), we can write $x = yz$, with $y \in \Gamma$ and $z \in D(\mathcal{O}_{L_w})$. Then
$$
z \xi(\sigma) \sigma(z)^{-1} = y^{-1} \sigma(y) \in D(\mathcal{O}_{L_w}) \cap \Gamma = \{ 1 \}.
$$
Consequently, $\xi(\sigma) = z^{-1} \sigma(z)$, showing that $\xi$ represents the trivial cohomology class in $H^1(L_w/K_v , D(\mathcal{O}_{L_w}))$.
\end{proof}

Next, we consider the subgroup of {\it integral adeles}
$$
D(\mathbb{A}^{\infty}(K , V)) := \prod_{v \in V} D(\mathcal{O}_v)
$$
and the corresponding class set
$$
\mathrm{cl}(D, K, V) := D(\mathbb{A}^{\infty}(K, V)) \backslash D(\mathbb{A}(K , V)) / D(K).
$$
In \cite{CRR-Isr}, \cite{RR-tori}, we introduced

\vskip1.5mm

\noindent {\bf Condition (T)} {\it There exists a finite subset $S \subset V$ such that $$\vert \mathrm{cl}(D, K, V \setminus S) \vert = 1.$$}

\vskip1.5mm

The following result is \cite[Theorem 3.4]{RR-tori}.

\begin{thm}\label{T:ConT}
Let $K$ be a finitely generated field and $V$ be a divisorial set of places of $K$. Then any linear algebraic $K$-group $D$ whose connected component is a torus satisfies Condition $(\mathrm{T})$.
\end{thm}

\vskip2mm

\noindent {\it{Proof of Theorem \ref{T:1}.}} (i): By Proposition \ref{P:adeles}, we have $$\ker \lambda^i_{D, V, L/K} = \ker \theta^i_{L/K},$$ so it is enough to prove the finiteness of the latter. Applying Theorem \ref{T:ConT} to $D$ over $L$ and the divisorial set of places $V^L$ of $L$, we see that after deleting from $V$ a finite set of places (which can only make $\ker \theta^i_{L/K}$ larger), we can assume that $\vert \mathrm{cl}(D, L, V^L) \vert = 1$, i.e. $D(\mathbb{A}(L , V^L)) = D(\mathbb{A}^{\infty}(L , V^L)) D(L)$. We then have the short exact sequence
$$
1 \to E \longrightarrow D(\mathbb{A}^{\infty}(L , V^L)) \times D(L) \stackrel{\mu}{\longrightarrow} D(\mathbb{A}(L , V^L)) \to 1,
$$
where $\mu$ is the product map and $E := D(L) \cap D(\mathbb{A}^{\infty}(L , V^L))$, and we consider the
following fragment of the corresponding cohomological long exact sequence
{\small
$$
H^i(L/K , E) \longrightarrow H^i(L/K , D(\mathbb{A}^{\infty}(L , V^L))) \times H^i(L/K , D(L)) \stackrel{\mu^i}{\longrightarrow} H^i(L/K , D(\mathbb{A}(L , V^L))).
$$}
As we discussed in the proof of \cite[Proposition 3.2]{RR-tori}, the intersection $$T(L) \cap T(\mathbb{A}^{\infty}(L , V^L))$$ is a finitely generated group.  Since $E$ contains this intersection as a subgroup of finite index, it is itself finitely generated. Then the group $H^i(L/K , E)$ is finite (cf. \cite[Ch. IV, \S 6, Corollary 2]{ANT}), implying the finiteness of $\ker \mu^i$. On the other hand, we clearly have  $\{ 0 \} \times \ker \theta^i_{L/K} \subset \ker \mu^i$, and our claim follows.

\vskip1mm

(ii): We now give a noncommutative version of the above argument for $i = 1$,  which requires the following.

\begin{lemma}\label{L:Finite}
Let $G$ be a finite group and $\Lambda$ be a $G$-group. Assume that $\Lambda$ admits a finite index subgroup $\Theta \subset \Lambda$ that is a finitely generated abelian group. Then the set $H^1(G , \Lambda)$ is finite.
\end{lemma}
\begin{proof}
The intersection $\Theta' = \bigcap_{\lambda \in \Lambda} (\lambda \Theta \lambda^{-1})$ is a {\it normal} subgroup of $\Lambda$ of finite index. Being a subgroup of $\Theta$, the subgroup $\Theta'$ is itself a finitely generated group. Thus, replacing $\Theta$ by $\Theta'$, we may assume that $\Theta$ is normal in $\Lambda$. Similarly, replacing $\Theta$ by $\bigcap_{g \in G} g(\Theta)$, we may also assume that $\Theta$ is $G$-invariant. Let $$\varphi \colon
H^1(G , \Lambda) \longrightarrow H^1(G , \Lambda/\Theta)$$ be the canonical map. Since the set $H^1(G , \Lambda/\Theta)$ is finite, it is enough to show that for each $x \in H^1(G , \Lambda/\Theta)$, the fiber $\varphi^{-1}(x)$ is finite. There is nothing to prove if $\varphi^{-1}(x) = \emptyset$. Otherwise, we pick a cocycle $\zeta \in Z^1(G , \Lambda)$ so that the image of its cohomology class under $\varphi$ is $x$, and let ${}_{\zeta}\Theta$ denote the corresponding twisted group. Then there is a surjection of $H^1(G , {}_{\zeta}\Theta)$ onto $\varphi^{-1}(x)$ (see \cite[Ch. I, \S 5.5, Corollary 2]{Serre-GC}). But since ${}_{\zeta}\Theta$ is a finitely generated abelian group, $H^1(G , {}_{\zeta}\Theta)$ is finite (cf. \cite[Ch. IV, \S 6, Corollary 2]{ANT}), so our claim follows.
\end{proof}

Using twisting, we see that it is enough to prove the finiteness of $\ker \theta^1_{L/K}$ for any $D$ as in theorem. Furthermore, due to Theorem \ref{T:ConT},
we may assume that
\begin{equation}\label{E:1Y}
D(\mathbb{A}(L , V^L)) = D(\mathbb{A}^{\infty}(L , V^L)) D(L).
\end{equation}
Suppose that $x \in \ker \theta^1_{L/K}$ is represented by a cocycle $\zeta \in Z^1(L/K , T(L))$. Then there exists $g \in D(\mathbb{A}(L , V^L))$ such that
$$
\zeta(\sigma) = g^{-1} \sigma(g) \ \ \text{for all} \ \ \sigma \in \mathrm{Gal}(L/K).
$$
According to (\ref{E:1Y}), we can write $g = ab$ where $a \in D(\mathbb{A}^{\infty}(L , V^L))$ and $b \in D(L)$. Then
$$
a \zeta(\sigma) \sigma(a)^{-1} = b^{-1} \sigma(b) \in D(L) \cap D(\mathbb{A}^{\infty}(L , V^L)) =: E.
$$
Thus $x$ lies in the image of the map $H^1(L/K , E) \to H^1(L/K , D(L))$, and hence all of $\ker \theta^1_{L/K}$ lies in this image. As we pointed out in the proof of (i), $E$ contains a finitely generated abelian group as a subgroup of finite index. So, by Lemma \ref{L:Finite}, the set $H^1(L/K , E)$ is finite, and the finiteness of $\ker \theta^1_{L/K}$ follows. \hfill $\Box$

\vskip2mm

\noindent {\it Proof of Theorem \ref{T:2}.} Again, the use of twisting shows that it is enough to establish the finiteness of the kernel of the map
$$
\theta \colon H^1(K , D) \longrightarrow \prod_{v \in V} H^1(K_v , D)
$$
for any $D$ as in the statement of the theorem. We will derive this from Theorem \ref{T:1}(ii). Let $T = D^{\circ}$. Then the quotient $\Omega := D/T$ is finite, so the kernel of $$\kappa \colon H^1(K , \Omega) \longrightarrow \prod_{v \in V} H^1(K_v , \Omega)$$ is finite by Proposition \ref{P:1}. So, we can find a finite Galois extension $L/K$ that splits $T$ and for which the image of $\ker \kappa$ under the restriction map $H^1(K , \Omega) \to H^1(L , \Omega)$ is trivial (see the proof of Proposition \ref{P:1}). Let us show that then for the restriction map $\rho \colon H^1(K , D) \to H^1(L , D)$, we also have
\begin{equation}\label{E:1Z}
\rho(\ker \theta) = \{ 1 \}.
\end{equation}
The exact sequence
$$
1 \to T \longrightarrow D \longrightarrow \Omega \to 1,
$$
gives rise to the following commutative diagram
$$
\xymatrix{H^1(K,D) \ar[rr]^{\delta_K} \ar[d]_{\theta} & & H^1(K, \Omega) \ar[d]^{\kappa} \\ \displaystyle{\prod_{v \in V}} H^1(K_v, D) \ar[rr] & & \displaystyle{\prod_{v \in V}} H^1(K_v, \Omega)}
$$
Let $x \in \ker \theta$. We then conclude from the diagram that $\delta_K(x) \in \ker \kappa$, and hence by our construction,  the image of $x$ under the composite map
$$
H^1(K , D) \stackrel{\delta_K}{\longrightarrow} H^1(K , \Omega) \longrightarrow H^1(L , \Omega)
$$
is trivial. Consequently, in view of the commutative diagram
$$
\xymatrix{H^1(K, D) \ar[r]^{\delta_K} \ar[d]_{\rho} & H^1(K, \Omega) \ar[d] \\ H^1(L, D) \ar[r]^{\delta_L} & H^1(L, \Omega)}
$$
the image of $\rho(x)$ under the map $\delta_L \colon H^1(L , D) \to H^1(L , \Omega)$ is also trivial. Since $H^1(L , T)$ is trivial by Hilbert's Theorem 90, we conclude from the exact sequence
$$
H^1(L , T) \longrightarrow H^1(L , D) \longrightarrow H^1(L , \Omega).
$$
that $\rho(x)$ is trivial, proving (\ref{E:1Z}).

To conclude the argument, we consider the following commutative diagram, in which the rows are the inflation-restriction exact sequences:
$$
\xymatrix{1 \ar[r] & H^1(L/K, D) \ar[rr]^v \ar[d]_{\theta_L} & & H^1(K,D)  \ar[d]^{\theta} \ar[r]^{\rho} & H^1(L,D) \\ 1 \ar[r] & \displaystyle{\prod_{v \in V}} H^1(L_w/K_v, D) \ar[rr]^{\Upsilon} & & \displaystyle{\prod_{v \in V}} H^1(K_v, D) & &}
$$
In conjunction with (\ref{E:1Z}), the diagram yields  $\ker \theta = \upsilon(\ker \theta_L)$, and since $\ker \theta_L$ is finite by Theorem \ref{T:1}(ii), we conclude that $\ker \theta$ is finite, as required. \hfill $\Box$

\vskip1mm

\noindent {\it Proof of Theorem \ref{T:2A}.} Let $G$ be a reductive $K$-group, $T$ be a maximal $K$-torus of $G$, and $N = N_G(T)$ be its normalizer. The variety $\mathscr{T}$ of maximal tori of $G$ is a $K$-defined variety whose points $\mathscr{T}(F)$ over a field extension $F/K$ bijectively correspond to the maximal $F$-defined tori of $G$; note that $\mathscr{T}$ can be identified with the quotient $G/N$ (cf. \cite[2.4.8]{Pl-R}). Furthermore,
there is a map $\gamma_F \colon \mathscr{T}(F) \to H^1(F , N)$ that is defined as follows: given an $F$-defined maximal torus $T' \in \mathscr{T}(F)$, we pick an arbitrary $g \in G(F^{\mathrm{sep}})$ such that $T' = g T g^{-1}$, and then $\xi(\sigma) = g^{-1} \sigma(g)$ for $\sigma \in \mathrm{Gal}(F^{\mathrm{sep}}/F)$ defines a Galois 1-cocycle with values in $N(F^{\mathrm{sep}})$ and $\gamma_F$ sends $T'$ to the cohomology class of this cocycle. A standard argument shows that the fibers of $\gamma_F$ correspond to the $G(F)$-conjugacy classes of $F$-defined maximal tori of $G$, and the fiber over the trivial class is precisely the $G(F)$-conjugacy class of $T$.

Now, to prove the theorem we need to show that the image $\gamma_K(\mathscr{C}(T))$ is finite. For this we observe that this image is contained in
$$
\ker\left(H^1(K , N) \to \prod_{v \in V} H^1(K_v , N) \right).
$$
Since the connected component $N^{\circ}$ is $T$ (see Corollary 2 of 8.10 and Corollary 2 of 13.17 in \cite{BorelAG}), this kernel is finite by Theorem \ref{T:2}, and the finiteness of $\gamma_K(\mathscr{C}(T))$ follows. \hfill $\Box$

\section{Tori with good reduction: general set-up}\label{S:Funct}

We now turn to finiteness results for tori with good reduction over function fields of algebraic varieties defined over special fields.
In this section, we will describe a general formalism that extends the argument used in the proof of \cite[Theorem 1.1]{RR-tori}; for the reader's convenience we include all the details. After these preparations, we will formulate and prove our finiteness results in the most general form possible in the next section, obtaining, in particular, a proof of  Theorem \ref{T:3}. First, we will describe the required condition on the base field. In \cite[Ch. III, \S 4]{Serre-GC}, Serre introduced the following condition on a profinite group $\mathscr{G}$:

\vskip2mm

\noindent (F) \parbox[t]{15cm}{For any finite group $\Phi$, the set $\mathrm{Hom}_{\mathrm{cont}}(\mathscr{G} , \Phi)$ of
continuous group homomorphisms $\mathscr{G} \to \Phi$ is finite.}

\vskip2mm

\noindent (We recall that this condition is equivalent to the requirement that for each $n \geq 1$, the group $\mathscr{G}$ has finitely many open subgroups of index $n$.) Nowadays, groups satisfying condition (F) are often called {\it small}; note that every finitely generated profinite group is small. Furthermore, a field $k$ is said to be of {\it type} (F) if its absolute Galois group $\mathrm{Gal}(k^{\mathrm{sep}}/k)$ is small (we note that in his definition, Serre also requires $k$ to be perfect, but this is not needed in our context). Well-known examples of fields of type (F) are algebraically closed fields, finite fields, and finite extensions of the $p$-adic field $\mathbb{Q}_p$, but in fact, there are many others --- see the discussion in \cite[\S 2]{IR}.

To formulate our results in full generality, it is convenient to introduce a variant of condition (F). Let $p$ be either 1 or a prime number. We then consider the following condition on a profinite group $\mathscr{G}$:

\vskip2mm

\noindent $(\mathrm{F}(p))$ \parbox[t]{15cm}{For any finite group $\Phi$ \underline{of order prime to} $p$, the set $\mathrm{Hom}_{\mathrm{cont}}(\mathscr{G} , \Phi)$ is finite.}

\vskip2mm

\noindent We note that condition $(\mathrm{F}(1))$ coincides with (F), and that (F) implies $(\mathrm{F}(p))$ for all $p$. Given a profinite group $\mathscr{G}$, we let $\mathscr{N}$ denote the intersection of the kernels of all continuous homomorphisms $\varphi \colon \mathscr{G} \to \Phi$, where $\Phi$ is a finite group of order prime to $p$. Then $\mathscr{G}^{(p)} := \mathscr{G}/\mathscr{N}$ is referred to as the {\it maximal prime-to-$p$ quotient} of $\mathscr{G}$.
It follows from the definition that for any finite group $\Phi$ of order prime to $p$, we have
$$
\mathrm{Hom}_{\mathrm{cont}}(\mathscr{G} , \Phi) = \mathrm{Hom}_{\mathrm{cont}}(\mathscr{G}^{(p)} , \Phi).
$$
On the other hand, the image of any continuous homomorphism $\varphi \colon \mathscr{G}^{(p)} \to \Phi$ to a finite group $\Phi$ is a subgroup $\Phi' \subset \Phi$ of order prime to $p$. We conclude that $\mathscr{G}$ satisfies $(\mathrm{F}(p))$ if and only if $\mathscr{G}^{(p)}$ satisfies $(\mathrm{F})$. With these definitions in place, we are now ready to formulate the following.

\begin{thm}\label{T:tori-GR}
Let $\mathfrak{X}$ be a normal separated integral scheme of finite type over $\mathbb{Z}$ or over a field, let $K$ be the function field of $\mathfrak{X}$, and let $V$ be the set of discrete valuations of $K$ corresponding to the prime divisors on $\mathfrak{X}$. Assume that the fundamental group $\pi_1^{\text{\'et}}(\mathfrak{X})$ satisfies condition $(\mathrm{F}(p))$. Then for every $d \geq 1$, there exist finitely many $K$-isomorphism classes of $d$-dimensional $K$-tori $T$ that have good reduction at all $v \in V$ and for which the degree $[K_T : K]$ of the splitting field is prime to $p$.
\end{thm}

\noindent (Since $\mathfrak{X}$ is irreducible, the isomorphism class of the fundamental group does not depend on the choice of a geometric point, so, to simplify notation, we do not specify the geometric point in the statement. We also note that when $p = 1$, i.e. when the fundamental group is small, the conclusion applies to all $d$-dimensional tori without any restrictions on the degree of the splitting field.)

For the argument, we fix an algebraic closure $\overline{K}$ of $K$ and let $\overline{y} \colon {\rm Spec}(\overline{K}) \to \mathfrak{X}$ be the corresponding geometric point of $\mathfrak{X}.$ We also fix a separable closure $K^{\rm sep} \subset \overline{K}.$
Furthermore, we denote by $K_V/K$ the maximal subextension of $K^{\rm sep}$ that is unramified at all $v \in V$. With these notations, we have the following statement, which is a key element in the proof of Theorem \ref{T:tori-GR}.


\begin{prop}\label{P-HigherHM}
The extension $K_V/K$ is Galois and $\Ga(K_V/K)$ satisfies condition $(\mathrm{F}(p))$.
\end{prop}

We begin the proof with the following

\begin{lemma}\label{L-ZNPurity}
With the above notations, let $K_{\mathfrak{X}}/K$ be the compositum of all finite subextensions $L/K$ of $K^{\rm sep}$ such that the normalization of $\mathfrak{X}$ in $L$ is \'etale over $\mathfrak{X}$. Then $K_{\mathfrak{X}} = K_V.$
\end{lemma}
\begin{proof}
It follows from the definitions that we have the inclusion $K_{\mathfrak{X}} \subset K_V.$ To show the reverse inclusion, suppose that $L/K$ is a finite subextension of $K^{\rm sep}$ that is unramified at all $v \in V$, and let $\mathfrak{Y}$ be the normalization of $\mathfrak{X}$ in $L.$ Then by assumption, $\mathfrak{Y} \to \mathfrak{X}$ is finite \'etale over each codimension 1 point of $\mathfrak{X}$. The Zariski-Nagata purity theorem, whose statement we recall below for completeness, then implies that $\mathfrak{Y}$ is \'etale over $\mathfrak{X}$, hence $L \subset K_{\mathfrak{X}}.$
\end{proof}

\begin{thm}\label{T-ZariskiNagata}{\rm (Zariski-Nagata purity theorem)} Let $\varphi \colon Y \to S$ be a finite surjective morphism of integral schemes, with $Y$ normal and $S$ regular. Assume that the fiber of $Y_P$ of $\varphi$ above each codimension 1 point of $S$ is \'etale over the residue field $\kappa(P).$ Then $\varphi$ is \'etale.

\end{thm}

\noindent (See, for example, \cite[Theorem 5.2.13]{Szamuely} for the statement and related discussion and \cite[Exp. X, Th\'eor\`eme 3.4]{SGA2} for a detailed proof.)

\vskip2mm

\noindent {\it Proof of Proposition \ref{P-HigherHM}.} Since $\pi^{\text{\'et}}(\mathfrak{X})$ is assumed to satisfy $(\mathrm{F}(p))$, our claim follows immediately from the well-known facts that $K_{\mathfrak{X}}/K$ is a Galois extension, and $\Ga(K_{\mathfrak{X}}/K)$ is canonically isomorphic to the fundamental group $\pi_1^{\text{\'et}}(\mathfrak{X}, \bar{y})$ for the geometric point $\bar{y} \colon {\rm Spec}(\overline{K}) \to \mathfrak{X}$ (see, for example, \cite[Proposition 5.4.9]{Szamuely}). $\Box$ 

\vskip2mm

\noindent We now turn to

\vskip1mm

\noindent {\it Proof of Theorem \ref{T:tori-GR}.} Recall that given a $d$-dimensional torus $T$, the action on the group of characters $X(T)$ gives rise to
a continuous representation $$\rho \colon \Ga (K^{\mathrm{sep}}/K) \to {\rm GL}_d(\Z)$$ of the absolute Galois group. The kernel $\ker \rho$ is the subgroup
$\Ga(K^{\mathrm{sep}}/K_T)$ corresponding to the splitting field $K_T$ of $T$, so the image is isomorphic to the Galois group $\Ga(K_T/K)$.
Moreover, we thereby obtain a bijection between the $K$-isomorphism classes of such tori and the equivalence classes of such representations (see, for example, \cite[\S 2.2.4]{Pl-R}). On the other hand, a $K$-torus $T$ has good reduction at a place $v$ of $K$ if and only if $T \times_K K_v$ splits over an unramified extension of the completion $K_v$ (see, for example, \cite[1.1]{NX}); in other words, the extension $K_T/K$ is unramified at $v$. This means that
the $K$-isomorphism classes of $d$-dimensional $K$-tori having good reduction at all $v \in V$ are in bijection with the equivalence classes of continuous representations $\rho \colon \Ga(K_V/K) \to {\rm GL}_d(\Z).$

Next, by reduction theory (see \cite[Theorem 4.9]{Pl-R}), the group $\mathrm{GL}_d(\mathbb{Z})$ has finitely many conjugacy classes of finite subgroups; let $\Phi_1, \ldots , \Phi_r$ be a complete set of representatives of the conjugacy classes of subgroups that have order prime to $p$. Then the isomorphism class of a given $K$-torus $d$-dimensional $T$ that has good reduction at all $v \in V$ and for which the degree $[K_T : K]$ is prime to $p$ corresponds to the equivalence class of a representation $\rho \colon \Ga(K_V/K) \to {\rm GL}_d(\Z)$ whose image is one of the $\Phi_i$'s. But according to Proposition \ref{P-HigherHM}, for each $i \in \{1, \ldots , r\}$, there are finitely continuous homomorphisms $\Ga(K_V/K) \to \Phi_i$. It follows that there are finitely
many equivalence classes of the relevant representations $\rho \colon \Ga(K_V/K) \to {\rm GL}_d(\Z)$, hence finitely many isomorphism classes of tori having good reduction at all $v \in V$. $\Box$

\vskip5mm



\section{Tori with good reduction over function fields}

As we observed in \cite{RR-tori}, the finiteness results for the isomorphism classes of tori of a given dimension having good reduction at a certain natural set of places of the base field cannot be extended from characteristic zero to positive characteristic without additional assumptions. The precise conditions for function fields in all characteristics are given in the next statement, which will then be used to treat the case of finitely generated fields in positive characteristic.
%
%
%
%
We let $p$ denote the characteristic exponent of the base field $k$, i.e. $p = 1$ if $\mathrm{char}\: k = 0$ and $p = \mathrm{char}\: k$ otherwise.
\begin{thm}\label{T:Z1}
Let $K = k(X)$ be the function field of a normal geometrically integral variety defined over a field $k$ of type $(\mathrm{F})$, and let $V$ be the set of discrete valuations of $K$ associated with the prime divisors of $X$.

\vskip2mm

\noindent \ {\rm (i)} \parbox[t]{16cm}{If $X$ is complete, then for each $d \geq 1$, the set of $K$-isomorphism classes of $d$-dimensional $K$-tori that have good
reduction at all $v \in V$ is finite.}

\vskip2mm

\noindent {\rm (ii)} \parbox[t]{16cm}{In the general case, for each $d \geq 1$, the set of $K$-isomorphism classes of $d$-dimensional $K$-tori $T$ that have good reduction at all $v \in V$ and for which the degree $[K_T : K]$ of the splitting field is prime $p$ is finite.}

\end{thm}


If $\mathrm{char}\: k = 0$, then part (ii) of the theorem applies to any $d$-dimensional torus $T$ having good reduction at all $v \in V$, thus yielding Theorem \ref{T:3}. On the other hand, any finitely generated field $K$ of characteristic $p > 0$ can be realized as the function field $k(X)$ of a geometrically integral normal variety $X$ over a finite field $k$. The choice of such a realization gives rise to a divisorial set $V$ of discrete valuations of $K$ that are associated to prime divisors of $X$. If $X$ is chosen to be complete, the corresponding $V$ is also called {\it complete}. Since finite fields are of type (F), we obtain the following.
\begin{cor}
Let $K$ be a finitely generated field of characteristic $p > 0$, and let $V$ be a divisorial set of places of $K$.

\vskip1mm

\noindent {\rm (i)} \parbox[t]{16cm}{If $V$ is complete, then for any $d \geq 1$, the set of $K$-isomorphism classes of $d$-dimensional $K$-tori
that have good reduction at all $v \in V$ is finite.}

\vskip1mm

\noindent {\rm (ii)} \parbox[t]{16cm}{In the general case, for any $d \geq 1$, the set of $K$-isomorphism classes of $d$-dimensional $K$-tori $T$ that have good reduction at all $v \in V$ and for which the degree $[K_T : K]$ of the splitting field is prime to $p$ is finite.}

\end{cor}

\vskip2mm

We will derive both parts of Theorem \ref{T:Z1} from Theorem \ref{T:tori-GR}. Its application in the situation of part (i) is justified
by the following statement.
\begin{prop}\label{P:Z1}
Let $X$ be a complete normal geometrically integral variety over a field $k$ of type $(\mathrm{F})$. Then the \'etale fundamental group $\pi_1^{\text{\'et}}(X)$ (with respect to any geometric point) satisfies $(\mathrm{F})$, i.e. is small.
\end{prop}
\begin{proof}
We fix a separable closure $k^{\mathrm{sep}}$ of $k$, set $X^{s} = X \times_{\mathrm{Spec}\: k} \mathrm{Spec}\: k^{\mathrm{sep}}$, and pick a geometric point $\bar{x}$ of $X^s$. We then have the following standard exact sequence of profinite groups
\begin{equation}\label{E:Z1}
1 \to \pi_1^{\text{\'et}}(X^s, \bar{x}) \to \pi_1^{\text{\'et}}(X, \bar{x}) \to \Ga(k^{\mathrm{sep}}/k) \to 1.
\end{equation}
(see \cite[Exp. IX, Th. 6.1]{SGA1}) By our assumption, $\Ga(k^{\mathrm{sep}}/k)$ is small. Furthermore, since $X$ is complete, the fundamental group $\pi_1(X^s, \bar{x})$ is topologically finitely generated for $k$ of any characteristic (see Exp. X, Th. 2.9, combined with Exp. IX, 4.10, in \cite{SGA1}),
and hence is small as well. Applying \cite[Lemma~2.7]{HH-small}, we conclude that $\pi^{\text{\'et}}_1(X, \bar{x})$ is small.
\end{proof}

\vskip1mm

To treat part (ii) of the theorem using a similar strategy, we need information about finite generation of the prime-to $p$ part $\pi^{\text{\'et}}_1(X^s, \bar{x})^{(p)}$ of the fundamental group. Raynaud \cite{Ray} established this fact for any scheme $\mathscr{X}$ of finite type over a separably closed field
$\mathscr{K}$ assuming that all schemes of finite type over an algebraic closure of $\mathscr{K}$ that have dimension $\leq \dim \mathscr{X}$ are ``fortement d\'esingularisable" as defined in \cite[Exp. I, 3.1.5]{SGA5} (it should be noted that this assumption is known to hold if either $p = 1$ or $\dim \mathscr{X} \leq 2$). Later, by working with alterations, Orgogozo \cite{Orgogozo} established this fact unconditionally. Recently, a new proof of the finite generation
of the prime-to $p$ part of the fundamental group of a smooth quasi-projective variety over an algebraically closed field was given in \cite[Proposition 3.1]{Esn}, which, by \cite[Exp. IX, 4.10]{SGA1}, implies the same result over separably closed fields (we note that in \cite{Esn}, it is in fact shown that the tame fundamental group is finitely presented when there is a good compactification, from which the finite presentation of the prime-to $p$ part of the fundamental group is deduced).
In terms of proving part (ii), we can always replace $X$ with an open subset as this will result in a smaller set of places $V$. Thus, we can assume that $X$ is affine and smooth. In this case, part (ii)
follows from Theorem \ref{T:tori-GR} and the proposition below.
\begin{prop}\label{P:Z2}
Let $X$ be a geometrically integral smooth affine variety over a field $k$ whose absolute Galois group satisfies $(\mathrm{F}(p))$. Then the \'etale fundamental group $\pi_1^{\text{\'et}}(X)$ also satisfies  $(\mathrm{F}(p))$.
\end{prop}
The proof is similar to the proof of Proposition \ref{P:Z1}, so we will use the same notations and utilize the exact sequence (\ref{E:Z1}).
As we discussed above, it follows from \cite[Proposition 3.1]{Esn} that the maximal prime-to-$p$ quotient $\pi_1^{\text{\'et}}(X^s, \bar{x})^{(p)}$ is finitely generated, hence satisfies $(\mathrm{F})$. Thus, the fundamental group $\pi_1^{\text{\'et}}(\overline{X}, \bar{x})$ itself satisfies $(\mathrm{F}(p))$. Since the absolute Galois group $\mathrm{Gal}(k^{\mathrm{sep}}/k)$ satisfies $(\mathrm{F}(p))$ by our assumption, our claim is a consequence of the following.
\begin{lemma}
Let $1 \to \mathscr{E} \longrightarrow \mathscr{G} \stackrel{\alpha}{\longrightarrow} \mathscr{H} \to 1$ be an extension of profinite groups. If $\mathscr{E}$ and $\mathscr{H}$ satisfy $(\mathrm{F}(p))$, then so does $\mathscr{G}$.
\end{lemma}
\begin{proof}
Write the maximal prime-to-$p$ quotient $\mathscr{G}^{(p)}$ in the form $\mathscr{G}^{(p)} = \mathscr{G}/\mathscr{N}$, where $\mathscr{N}$ is the intersection of the kernels of all continuous homomorphisms $\varphi \colon \mathscr{G} \to \Phi$, with $\Phi$ a finite group of order prime to $p$. Then  $\mathscr{G}^{(p)}$ is contained in the exact sequence
$$
1 \to \mathscr{E}/ (\mathscr{E} \cap \mathscr{N}) \longrightarrow \mathscr{G}^{(p)} \longrightarrow \mathscr{H}/\alpha(\mathscr{N}) \to 1.
$$
It is easy to see that $\mathscr{E}/ (\mathscr{E} \cap \mathscr{N})$ and $\mathscr{H}/\alpha(\mathscr{N})$ are quotients of the maximal prime-to-$p$ quotients
$\mathscr{E}^{(p)}$ and $\mathscr{H}^{(p)}$, respectively. Since $\mathscr{E}$ and $\mathscr{H}$ satisfy $(\mathrm{F}(p))$, the groups $\mathscr{E}^{(p)}$ and $\mathscr{H}^{(p)}$ satisfy $(\mathrm{F})$, and therefore so do the groups $\mathscr{E}/ (\mathscr{E} \cap \mathscr{N})$ and $\mathscr{H}/\alpha(\mathscr{N})$.
Then by \cite[Lemma 2.7]{HH-small}, the group $\mathscr{G}^{(p)}$ also satisfies $(\mathrm{F})$, and therefore $\mathscr{G}$ satisfies $(\mathrm{F}(p))$.
\end{proof}

\section[Properness of the global-to-local map over function\\\mbox{} fields of algebraic varieties]{Properness of the global-to-local map over function fields of algebraic varieties}\label{S:local-global}


In this section, $K$ will denote the function field $k(X)$ of a normal geometrically integral variety $X$ over a field $k$. Let $V$ be the set of discrete valuations of $K$ associated with the prime divisors of $X$.

\vskip1mm

\noindent {\it Proof of Theorem \ref{T:4}.} Assuming that the base field $k$ has characteristic zero and is of type $(\mathrm{F})$, we will prove that for any algebraic $K$-torus $T$, the Tate-Shafarevich group
$$
\Sha^1(T,V) = \ker \left(H^1 (K,T) \to \prod_{v \in V} H^1(K_v, T) \right)
$$
is finite. One can give an argument based on the method developed in the second proof of \cite[Theorem 1.2]{RR-tori}.
To avoid repetition, however, we will describe a slightly different approach that, in the case of finitely generated fields, was suggested by Colliot-Th\'el\`ene. We begin by fixing some notations. Let $U \subset X$ be a smooth open affine subvariety such that $T$ extends to a torus $\mathbb{T}$ over $U$ and let $V_U$ be the set of discrete valuations of $K$ corresponding to the points of codimension 1 of $U$. For each such $x \in U$, let $v_x$ be the associated valuation of $K$. Denote by $\mathcal{O}_{x}$ the local ring of $U$ at $x$ and by $\hat{\mathcal{O}}_{v_x}$ its completion with respect to $v_x.$ Note that the fraction field of the latter is simply the completion $K_{v_x}$ of $K$ with respect to $v_x.$ According to \cite[Proposition 2.2]{CTS78}, the natural maps
$$
\he^1(\hat{\mathcal{O}}_{v_x}, \mathbb{T}) \to H^1(K_{v_x}, T)
$$
are injective for all $v_x \in V_U.$ Consider the global-to-local map
$$
\lambda_{T, V_U} \colon H^1(K,T) \to \prod_{v_x \in V_U} H^1(K_{v_x}, T).
$$
We then have the following general statement that immediately implies Theorem \ref{T:4}.

\begin{prop}\label{P-ToriSha}
The inverse image of $\displaystyle{\prod_{v_x \in V_U}} \he^1(\hat{\mathcal{O}}_{v_x}, \mathbb{T})$ in $H^1(K,T)$ under the map $\lambda_{T,V_U}$ is finite.
\end{prop}
\begin{proof}
Set
$$
R = \lambda_{T, V_U}^{-1} \left(\prod_{v_x \in V_U} \he^1(\hat{\mathcal{O}}_{v_x}, \mathbb{T})\right) \subset H^1(K,T)
$$
and let $\xi \in R.$ First, according to \cite[Lemma 4.1.3]{Harder} (see also \cite[Lemma 4.1]{CTPS}), the fact that the $v_x$-component of $\lambda_{T, V_U}(\xi)$ is contained in the image of $\he^1(\hat{\mathcal{O}}_{v_x}, \mathbb{T}) \to H^1(K_{v_x}, T)$ implies that $\xi$ is contained in the image of the natural map $\he^1({\mathcal{O}}_{x}, \mathbb{T}) \to H^1(K, T)$, which is injective by \cite[Proposition 2.2]{CTS78}. Thus, $\xi$ lies in the image of $\he^1({\mathcal{O}}_{x}, \mathbb{T}) \to H^1(K, T)$ for all codimension 1 points $x \in U$, and hence by purity, $\xi$ is contained in the image of the natural map
$$
\he^1(U, \mathbb{T}) \to H^1(K,T)
$$
(see \cite[Proposition 4.1]{CTS78} and \cite[Corollaire 6.9]{CTS79}). On the other hand, since $k$ is of type (F), the image of the latter is finite by \cite[Proposition 3.3]{CTGP}. It follows that $R$ is finite, as claimed.
\end{proof}


Our next statement treats the question of the properness of the global-to-local map in the Galois cohomology of finite Galois $K$-modules in all characteristics. We let $p$ denote the characteristic exponent of $k$.
\begin{prop}\label{P:Finite2}
Let $K = k(X)$ be the function field of a geometrically integral normal variety $X$ defined over a field $k$, $V$ be the set of discrete valuations associated
with the prime divisors of $X$, and $\Omega$ be a finite (but not necessarily commutative) Galois module. Then in each of the following situations

\vskip2mm

{\rm (1)} $\dim X \geq 2$,

{\rm (2)} $X$ is a projective curve,

{\rm (3)} \parbox[t]{15cm}{$X$ is an arbitrary curve, but the order of $\Omega$ is prime to the characteristic exponent $p$ of $k$}

\vskip2mm

\noindent the map
$$
H^1(K , \Omega) \stackrel{\kappa}{\longrightarrow} \prod_{v \in V} H^1(K_v , \Omega)
$$
is proper.
\end{prop}
\begin{proof}
As in the proof of Proposition \ref{P:1}, by using twisting, it is enough to show that $\ker \kappa$ is finite, and then one can assume (which we will) without loss of generality that the Galois action on $\Omega$ is trivial. Under this assumption, one actually proves that in case (1), the kernel $\ker \kappa$ is trivial --- this is derived from Proposition \ref{P:FFF3} by repeating verbatim the argument used to derive the corresponding statement in the proof of Proposition \ref{P:1} from Proposition \ref{P:2}.

\vskip1mm

(2): Let $K_V$ that be the compositum of all finite Galois extensions contained in $K^{\mathrm{sep}}$ that are unramified at all $v \in V$. By Lemma \ref{L-ZNPurity}, the field $K_V$ coincides with the compositum $K_X$ of all finite subextensions $L$ of $K^{\mathrm{sep}}$ such that the normalization of $X$ in $L$ is etale over $X$. Since $X$ is projective, as we noted in the proof of Proposition \ref{P:Z1}, the \'etale fundamental group $\mathscr{H} = \pi_1^{\text{\'et}}(X^s)$  of $X^s = X \times_{\mathrm{Spec}\: k} \mathrm{Spec}\: k^{\mathrm{sep}}$
is topologically finitely generated, hence satisfies condition $(\mathrm{F})$. Let $\widetilde{\mathscr{H}}$ be the absolute Galois group of $k^{\mathrm{sep}}(X)$, and let $\widetilde{M}$ be the fixed subfield for the intersection of the kernels of all continuous homomorphisms $\tilde{\chi} \colon \widetilde{\mathscr{H}} \to \Omega$ of the form $\tilde{\chi} = \chi \circ \nu$, where $\nu \colon \widetilde{\mathscr{H}} \to \mathscr{H}$ is the canonical epimorphism and $\chi \colon \mathscr{H} \to \Omega$ is a continuous homomorphism. Since the set $\mathrm{Hom}_{\mathrm{cont}}(\mathscr{H} , \Omega)$ is finite, $\widetilde{M}$ is a finite Galois extension of $k^{\mathrm{sep}}(X)$ having the following property: if $L/K$ is a finite Galois extension that is unramified at all $v \in V$ and the Galois group of which is isomorphic to a subgroup of $\Omega$, then $L \subset \widetilde{M}$. We pick a finite Galois extension $M$ of $K$ contained in $\widetilde{M}$ so that $\widetilde{M} = M \cdot k^{\mathrm{sep}}$. We then fix $v_0 \in V$ and its extension $w_0$ to $M$, and let $E$ denote the maximal separable extension of $k$ contained in the residue field $M^{(w_0)}$.

Now let $\chi \colon \mathscr{G} \to \Omega$ be a continuous homomorphism of the absolute Galois group $\mathscr{G} = \mathrm{Gal}(K^{\mathrm{sep}}/K)$ that lies in $\ker \kappa$,  and let $L_{\chi} $ be the fixed field of $\ker \chi$. To prove the finiteness of $\ker \kappa$, it is enough to show that for any such $\chi$, the field $L_{\chi}$ is contained in the finite extension $ME$ of $K$. In any case, $L = L_{\chi}$ is a finite Galois extension of $K$ whose Galois group is isomorphic to a subgroup of $\Omega$ and which satisfies $L_w = K_v$ for all $v \in V$ and $w \vert v$. In particular, $L/K$ is unramified at all $v \in V$, and hence $L \subset \widetilde{M}$. Because of the canonical isomorphism of Galois groups $$\mathrm{Gal}(\widetilde{M}/M) \simeq \mathrm{Gal}(k^{\mathrm{sep}}/(k^{\mathrm{sep}} \cap M)),$$ we conclude that there is a finite separable extension $\ell/k$ such that $F=ML$ coincides with  $M\ell$. Let $u_0$ be the extension of $w_0$ to $F$. Since the completion of $L$ with respect to the restriction of $u_0$ coincides with $K_{v_0}$, we see that $F_{u_0} = M_{w_0}$, implying the equality $F^{(u_0)} = M^{(w_0)}$ of the residue fields. On the other hand, $\ell \subset F^{(u_0)}$ hence $\ell \subset E$. Thus, $L \subset ML = M\ell \subset ME$, as required.

\vskip1mm

(3): Since in this case $\Omega$ is assumed to be of order prime to $p$, for the argument used to treat case (2) to work in this situation one only needs
to make sure that the prime-to-$p$ part $\pi^{\text{\'et}}_1(X^s)^{(p)}$ of the fundamental group is finitely generated, which was known for any  curve $X$  (see \cite[Exp. XIII, Cor. 2.12]{SGA1}) even before \cite{Esn}.
\end{proof}


\noindent {\bf Remark 6.3.} In the case where $X$ is an affine curve over a field $k$ of characteristic $p > 0$ and $\Omega$ is a finite Galois module of order divisible by $p$ over the function field $K = k(X)$, which is excluded in the proposition, the finiteness assertion may be false. Indeed, let $k$ be an algebraically closed field of characteristic $p$ and $X = \mathbb{A}^1_k$, and let  $\Omega = \mathbb{Z}/p\mathbb{Z}$ be the trivial Galois module over the field $K = k(X)$ of rational functions. As above, we let $V$ denote the set of places of $K$ corresponding to the closed points of $X$. If $L/K$ is a Galois extension of degree $p$ corresponding to some Artin-Schreier cover $Y \to X$, then any $v \in V$ is unramified in $L$, implying that $L_w = K_v$ for $w \vert v$ since the residue
field $K^{(v)} = k$ is algebraically closed. This means that a character $\chi \colon \mathscr{G} \to \Omega$ of the absolute Galois group $\mathscr{G} = \mathrm{Gal}(K^{\mathrm{sep}}/K)$ whose kernel is the subgroup corresponding to $L$ lies in the kernel of the global-to-local map $\kappa \colon H^1(K , \Omega) \to \prod_{v \in V} H^1(K_v , \Omega)$. Since $X$ has infinitely many distinct Artin-Schreier covers, $\kappa$ is {\it not} proper in this case.

\vskip3mm

\addtocounter{thm}{1}

With $K$ and $V$ as in Proposition \ref{P:Finite2}, we have the following statement.

\begin{prop}\label{P:Finite3}
Let $G$ be a connected reductive $K$-group. Fix a maximal $K$-torus $T$ of $K$ and let $\mathscr{C}(T)$ be the set of all maximal $K$-tori $T'$ of $G$ such that $T$ and $T'$ are $G(K_v)$-conjugate for all $v \in V$. Then, with the exception of the following case

\vskip2mm

\noindent $(*)$ \parbox[t]{15cm}{$X$ is an affine curve and the order of the Weyl group $W(G , T)$ is divisible by $p$,}

\vskip2mm

\noindent $\mathscr{C}(T)$ consists of finitely many $K$-isomorphism classes.
\end{prop}

\begin{proof}
We will freely use the notations introduced in the proof of Theorem \ref{T:2A} --- see \S\ref{S:Proofs}. We will identify the Weyl group $W = W(G , T)$ with the quotient $N/T$, and for a field extension $F/K$ consider the natural map $\delta_F \colon H^1(F , N) \to H^1(F , W)$. It is well-known that if $T_1 , T_2 \in \mathscr{T}(F)$ are such that
\begin{equation}\label{E:equal}
\delta_F(\gamma_F(T_1)) = \delta_F(\gamma_F(T_2)),
\end{equation}
then $T_1$ and $T_2$ are $F$-isomorphic. Indeed, using twisting, one easily reduces the argument to the case where $T_1 = T$. To be consistent with the notations used in the proof of Theorem \ref{T:2A}, set $T' = T_2$, pick $g \in G(F^{\mathrm{sep}})$ such that $T' = g T g^{-1}$, and define $\xi(\sigma) = g^{-1} \sigma(g)$ for $\sigma \in \mathrm{Gal}(F^{\mathrm{sep}}/F)$. Then $\gamma_F(T')$ is the class of the cocycle $\xi = \{ \xi(\sigma) \}$ in $H^1(F , N)$. Condition (\ref{E:equal}) yields
that $\delta_F(\xi)$ is trivial, so the exact sequence
$$
H^1(F , T) \longrightarrow H^1(F , N) \stackrel{\delta_F}{\longrightarrow} H^1(F , W)
$$
tells us that $\xi$ is equivalent to a cocycle with values in $T(F^{\mathrm{sep}})$. This means that the element $g$ can be chosen so that $\xi(\sigma) = g^{-1} \sigma(g)$ belongs to $T(F^{\mathrm{sep}})$ for all $\sigma \in \mathrm{Gal}(F^{\mathrm{sep}}/F)$. Using the fact that $T$ is commutative, one can easily see that the isomorphism $T \to T'$, $x \mapsto g x g^{-1}$, is defined over $F$.

Having recalled these facts, we are now ready to complete the proof of the proposition. It is enough to prove that the set $\delta_K(\gamma_K(\mathscr{C}(T)))$ is finite. By construction, $\gamma_{K_v}(\mathscr{C}(T))$ consists only of the trivial class for any $v \in V$. So, using the commutative diagram
$$
\xymatrix{H^1(K,N) \ar[rr]^{\delta_K} \ar[d] & & H^1(K, W) \ar[d]^{\kappa} \\ \displaystyle{\prod_{v \in V}} H^1(K_v, N) \ar[rr]^{\prod \delta_{K_v}} & & \displaystyle{\prod_{v \in V}} H^1(K_v, W)}
$$
we see that $\delta_K(\gamma_K(\mathscr{C}(T))) \subset \ker \kappa$. Since we exclude the case $(*)$, the finiteness of $\ker \kappa$ follows from Proposition \ref{P:Finite2}, completing the argument.
\end{proof}

\vskip1mm

\noindent {\bf Remark 6.5.} As in Remark 6.3, suppose $k$ is an algebraically closed field of characteristic $p > 0$, and let $K = k(X)$ be the function field of $X = \mathbb{A}^1_k$ and $V$ be the set of discrete valuations of $K$ associated with the closed points of $X$. We have seen that $K$ has an infinite family $L_1, L_2, \ldots $ of degree $p$ cyclic extensions corresponding to the Artin-Schreier covers of $X$, and that for every $i$, we have $L_{i w} = K_v$ for all $v \in V$ and $w \vert v$. Then the corresponding maximal tori $T_i = \mathrm{R}_{L_i/K}(\mathbb{G}_m)$ of $G = \mathrm{GL}_p$ are pairwise non-isomorphic over $K$ (since they have different splitting fields). On the other hand, for any $v \in V$, all of them become split over $K_v$, hence are $G(K_v)$-conjugate. This demonstrates that the case $(*)$ in Proposition \ref{P:Finite3} is an honest exception.

\section[Simple groups with good reduction over the function\\\mbox{} fields of complex surfaces]{Simple groups with good reduction over the function fields of complex surfaces}\label{S:CSurf}

The goal of this section is to prove Theorem \ref{T:CSurf}.
The key ingredient in the proof is the following fundamental result on Serre's Conjecture II (cf. \cite{BP}, \cite[Theorem 1.2]{CTGP}, \cite{Gille}, \cite[Theorem 1.6]{dJHS}).
\begin{thm}\label{T-SerreSurf}
Let $k$ be an algebraically closed field of characteristic 0 and $K = k(S)$ be the function field of a surface $S$ over $k$. Then $H^1(K, G) = 1$ for any absolutely almost simple simply connected $K$-group $G$.
\end{thm}

Let $\mathcal{G}$ be an absolutely almost simple simply connected group over a field $\mathcal{K}$, and let $\mathcal{L}$ be the minimal Galois extension of $\mathcal{K}$ over which $\mathcal{G}$ becomes an inner form of a split group. It is well-known that the degree $[\mathcal{L} : \mathcal{K}]$ can only be $1$, $2$, $3$, or $6$. Furthermore, if $\mathcal{G}$ has good reduction at a discrete valuation $v$ of $\mathcal{K}$, then the extension $\mathcal{L}/\mathcal{K}$ is unramified at $v$. Thus, in the situation of Theorem \ref{T:CSurf}, if $G'$ is a $K$-form of $G$ that has good reduction at all $v \in V$, then the corresponding minimal Galois extension $L'$ of $K$ over which $G'$ becomes an inner form of a split group has degree $1$, $2$, $3$, or $6$ and is unramified at all $v \in V$. It follows that $L'$ corresponds to an \'etale cover of $S$ (cf. Lemma \ref{L-ZNPurity}). Since the \'etale fundamental group of $S$ is topologically finitely generated (see \cite[Exp. X, Th. 2.9]{SGA1}), there are only finitely many possibilities $L_1, \ldots , L_r$ for the field $L'$  as $G'$ runs through all possible $K$-forms of $G$ having good reduction at all $v \in V$. Let $G^{(i)}$ be the quasi-split form of $G$ associated with $L_i$. Then any $K$-form $G'$ of $G$ with good reduction at all $v \in V$ is an {\it inner} form of one of the $G^{(i)}$'s. Thus, in order to prove Theorem \ref{T:CSurf}, it is enough to establish the following:

\vskip2mm

\noindent $(*)$ \parbox[t]{15cm}{{\it if $G$ is an absolutely almost simple simply connected quasi-split $K$-group, then the set of $K$-isomorphism classes of \underline{inner} forms of $G$ that have good reduction at all $v \in V$ is finite.}}

\vskip2mm

\noindent To prove this statement, we consider the exact sequence
$$
1 \to F \to G \to \overline{G} \to 1,
$$
where $F$ is the center of $G$ and $\overline{G}$ is the corresponding adjoint group, which gives rise to a coboundary map
$$
\theta_K \colon H^1(K, \overline{G}) \to H^2(K,F).
$$
Then Theorem \ref{T-SerreSurf}, in conjunction with a standard argument based on twisting  (cf. \cite[\S1.3.2]{Pl-R}, \cite[Ch. I, \S5]{Serre-GC}), shows that $\theta_K$ is injective. Now, let $G'$ be an inner form of $G$ that has good reduction at all $v \in V$. Then $G'$ is obtained from $G$ by twisting by a cocycle with values in the group of inner automorphisms of $G$, which we identify with $\overline{G}$. Let $\xi \in H^1(K, \overline{G})$ be the corresponding  cohomology class. For any $v \in V$, the residue field $K^{(v)}$ is the function field of a complex curve, hence has cohomological dimension $\leq 1$
by Tsen's theorem (cf. \cite[Ch. II, 3.3]{Serre-GC}). By our assumption, the reduction ${G'}^{(v)}$ is a connected reductive (in fact, absolutely almost simple) group. Applying Steinberg's theorem (see \cite[Ch. III, \S2.3, Theorem 1$'$]{Serre-GC}), we conclude that  the reduction $\underline{G'}^{(v)}$ of $G'$ at $v$ is quasi-split. Then, by Hensel's lemma, $G'$ is quasi-split over $K_v$. This being true for all $v \in V$, we see that $\xi$ lies in the kernel of the global-to-local map
$$
H^1(K, \overline{G}) \stackrel{\lambda_{\overline{G}, V}}{\longrightarrow} \prod_{v \in V} H^1(K_v, \overline{G}).
$$
So, it remains to be establish the finiteness of this kernel. For this we will use the commutative diagram
$$
\xymatrix{H^1(K, \overline{G}) \ar[r]^{\lambda_{\overline{G}, V}} \ar[d]_{\theta_K} & \displaystyle{\prod_{v \in V}} H^1(K_v, \overline{G}) \ar[d]^{\theta_{K_v}} \\ H^2(K, F) \ar[r]^{\lambda^2_{F,V}} & \displaystyle{\prod_{v \in V}} H^2(K_v, F)}.
$$
Clearly,
$$
\theta_K(\ker \lambda_{\overline{G}, V}) \subset \ker \lambda^2_{F, V}.
$$
As we remarked above, the map $\theta_K$ is injective, so it is enough to prove the finiteness of $\ker \lambda^2_{F, V}$. Replacing $S$ by an open subvariety if necessary, we can assume that $F$ extends to a group scheme $\mathbb{F}$ over $S$. Then a consequence of purity for smooth varieties over fields is that $\ker \lambda^2_{F,V}$ is contained in the image of the natural map $\he^2(S, \mathbb{F}) \to H^2(K, F)$ (see \cite[\S3.4]{CT-SB}). Since $k$ is algebraically closed, the group $\he^2(S, \mathbb{F})$ is finite according to \cite[Expos\'e XVI, Th\'eor\`eme 5.2]{SGA4}. (We note that a similar argument was used in the proof of \cite[Proposition 3.3]{CTGP}).
This completes the proof of $(*)$ and hence yields Theorem \ref{T:CSurf}.

\section*{Acknowledgements} We would like to thank Toshiro Hiranouchi for helpful correspondence and the anonymous referee for useful remarks. The second author was partially supported by NSF grant DMS-2154408.



\end{document}